\documentclass[12pt,a4paper]{article} 
\usepackage{graphicx}
\usepackage{multirow}
\usepackage{amsmath,amssymb,amsfonts}
\usepackage{amsthm}
\usepackage{mathrsfs}
\usepackage[title]{appendix}
\usepackage{xcolor}
\usepackage{textcomp}
\usepackage{manyfoot}
\usepackage{booktabs}
\usepackage{algorithm}
\usepackage{algorithmicx}
\usepackage{algpseudocode}
\usepackage{listings}
\usepackage{hyperref}

\usepackage{proof}
\usepackage{color}
\usepackage{graphicx}
\usepackage{pdflscape}

\newtheorem{theorem}{Theorem}
\newtheorem{lemma}{Lemma}
\newtheorem{corollary}{Corollary}

\theoremstyle{definition}
\newtheorem{definition}{Definition}

\newtheorem{example}{Example}

\newtheorem{proposition}{Proposition}

\newtheorem{ourclaim}{Claim}
\newtheorem{oldtheorem}{Theorem}

\newcommand{\PC}{\mathrm{BC}}

\newcommand{\BS}{\mathop{\hspace*{-.1em}\backslash\hspace*{-.1em}}}

\newcommand{\set}[2]{\{ \, #1 \mid #2 \, \}}

\begin{document}
\sloppy
\emergencystretch 3em

\title{Categorial grammars with unique category assignment}
\author{Maxim Vishnikin \and Alexander Okhotin}
\maketitle

\begin{abstract}
A categorial grammar
assigns one of several
syntactic categories to each symbol of the alphabet,
and the category of a string is then deduced 
from the categories assigned to its symbols
using two simple reduction rules.
This paper investigates a special 
class of categorial grammars,
in which only one category is assigned to each symbol,
thus eliminating ambiguity on the lexical level
(in linguistic terms, 
a unique part of speech is assigned to each word).
While unrestricted categorial grammars 
are equivalent to the context-free grammars,
the proposed subclass initially appears weak,
as it cannot define even some regular languages.
It is proved in the paper that it is actually powerful enough
to define a homomorphic encoding of every context-free language,
in the sense that for every context-free language $L$ over an alphabet $\Sigma$
there is a language $L'$ over some alphabet $\Omega$
defined by categorial grammar with unique category assignment
and a homomorphism $h \colon \Sigma \to \Omega^+$,
such that a string $w$ is in $L$ if and only if $h(w)$ is in $L'$.
In particular, in Greibach's hardest 
context-free language theorem,
it is sufficient to use a hardest language
defined by a categorial grammar with 
unique category assignment.
\end{abstract}

\section{Introduction}

Categorial grammars are a logical formalism 
for syntax specification
that dates back to Ajdukiewicz~\cite{Ajdukiewicz1935}.
In their modern form, categorial grammars
were defined by Bar-Hillel~\cite{Barhillel_categorial} in 1953,
and later Bar-Hillel et al.~\cite{BarhillelGaifmanShamir}
proved them to be equivalent to the context-free grammars.
Many variants and extensions of categorial 
grammars have been studied in the literature,
and particularly well-known among them
are the Lambek calculus~\cite{Lambek},
which was proved equivalent to the 
context-free grammars by Pentus~\cite{Pentus},
and combinatory categorial grammars,
which are equivalent to tree-adjoining 
grammars and a few related formalisms~\cite{WeirJoshi}.

Categorial grammars may assign multiple 
categories to a single symbol,
and this property is essential for 
their equivalence to the context-free grammars:
if there is only one category for each symbol,
then categorial grammars already cannot describe, for instance,
the set $a^+$ of all non-empty unary strings.
In linguistic descriptions, multiple category assignment
is necessary to describe homographs:
e.g., the English word ``bear''
may be either a noun 
or a verb,
and a categorial grammar for English
must accordingly assign both categories 
to the symbol representing this word,
so that the parse of every sentence involving the word ``bear''
chooses one of the two categories for
each occurrence of this word.

This paper investigates the case
in which only one category may be assigned to every symbol:
\emph{categorial grammars with unique category assignment}.
These grammars interpret each symbol in a unique way
(in linguistic terms, they assign only 
one part of speech to every word,
so that no homographs are possible).
Intuitively, this prohibits any ambiguous 
specifications in a categorial grammar:
the only remaining potential ambiguity is in the order of applying
the fixed reduction rules common to all categorial grammars.
This is a natural restriction to investigate,
and the $a^+$ example given above confirms
that these grammars define a proper 
subclass of the context-free languages.
The question is, how limited this subclass is,
and what is its computational complexity?

Categorial grammars with unique category assignment
were earlier studied by Kanazawa~\cite{Kanazawa},
who developed a learning algorithm for this class of grammars.
Kanazawa's~\cite{Kanazawa} learning algorithm
was actually applicable to a larger subclass of categorial grammars,
which are restricted to assign at most $k$ categories to each symbol,
where $k \geqslant 1$ is a constant.
Besides the algorithm, Kanazawa~\cite{Kanazawa} proved
that if the number of categories assigned to a symbol is bounded by $k+1$,
then the resulting grammars are strictly more powerful than grammars with the bound $k$.

This paper begins with a short introduction to categorial grammars,
presented in Section~\ref{section_definitions}.
The investigation of the expressive power
of grammars with unique category assignment
conducted in this paper
starts with apparently the first example of a string of categories
that can be reduced to two different categories:
the string constructed early in Section~\ref{section_simulation_multiple_assignment}
reduces either to $p/p$, or to $p \BS p$,
depending on the order of applying the reduction rules to the string.
Admittedly, one could hardly derive any practical use of this example,
but it is still valuable for demonstrating the very possibility
for a grammar with unique category assignment
to interpret a single string in two different ways,
in spite of interpreting each symbol unambiguously.
The next, more complicated example is a string of categories
that reduces to $p/p$ or to $t/t$, for any categories $p$ and $t$,
and the rest of Section~\ref{section_simulation_multiple_assignment}
gradually develops the method to construct strings of categories
that are reducible to a given finite set of categories,
under a certain elaborate encoding.

This method of simulating choice
is used in the next Section~\ref{section_encoding_grammars}
to encode an arbitrary context-free grammar
into a categorial grammar with a unique category assignment.
Given a grammar in the Greibach normal form,
for each symbol $a$, all rules $A_1 \to a \alpha_1$, \ldots $A_k \to a \alpha_k$ in the grammar featuring this symbol in the first position
are encoded into a long string of categories
that simulates the choice between these $k$ rules,
and can be reduced to each of the $k$ categories representing these rules.
Finally, a homomorphism $h$ mapping each $a$
to the corresponding string of categories
has the property that, for every string $w$,
the homomorphic image $h(w)$ is reducible to a category $S/S$
if and only if the original grammar defines the string $w$.

This shows that categorial grammars with unique category assignment
can actually define a homomorphic encoding of every context-free language.
This, in particular, yields
an improvement to Greibach's~\cite{Greibach_hardest} hardest language theorem:
it is sufficient to take a hardest language
defined by a categorial grammar with unique category assignment.

\section{Definitions}\label{section_definitions}

Let $\Sigma$ be a finite \emph{alphabet} of the language being defined,
its elements are called \emph{symbols}.
The set of non-empty strings over $\Sigma$
is denoted by $\Sigma^+$.
Throughout this paper,
any subset of $\Sigma^+$ is a \emph{language},
that is, all languages are assumed to be without the empty string.
Categorial grammars operate with categories constructed from primitive categories using left- and right-division operations.

\begin{definition}
Let $\PC = \{ p, q, r, \dots \}$ be a finite set of \emph{primitive categories}.
The set \emph{categories} is defined as follows.
\begin{itemize}
\item
    Every primitive category is a category.
\item
    If $A$ and $B$ are categories,
    then so are $(B \BS A)$ and $(A / B)$.
\end{itemize}
\end{definition}
The outer pair of brackets in a category is typically omitted,
as in $r \BS (p/q)$.
When strings of categories are written,
individual categories shall be separated by semicolons,
as in $r \BS (p/q); p; q/r$.

\begin{definition}[Ajdukiewicz~\cite{Ajdukiewicz1935}, Bar-Hillel et al.~\cite{BarhillelGaifmanShamir}]
A categorial grammar is a quadruple $(\Sigma, \PC, \triangleright, S)$,
where $\Sigma$ is a finite alphabet,
$\PC$ is a finite set of primitive categories,
$\triangleright$ is a finite relation of category assignments
of the form $a \triangleright A$, where $a \in \Sigma$ and $A$ is a category,
and $S$ is a category designated as the target category.

Strings of categories can be \emph{reduced} using the following \emph{reduction rules},
which transform two adjacent categories into one:
\begin{equation*}
	\infer{A}{A / B \qquad B}
	\qquad
	\infer{A}{ B \qquad B \BS A}
    \qquad (\text{for all categories } A, B)
\end{equation*}
Accordingly, a string of categories $\alpha$ is said to be
\emph{reducible in one step} to a string of categories $\beta$,
if $\alpha=\eta; A/B; B; \theta$ or $\alpha=\eta; B; B \BS A; \theta$,
while $\beta=\eta; A; \theta$, for some categories $A$ and $B$,
and for some strings of categories $\eta$ and $\theta$.
A string of categories $\alpha$ is said to be \emph{reducible} to $\beta$
if it can be transformed to $\beta$
by zero or more applications of the above rule.
\end{definition}

\begin{definition}
The language defined by a categorial grammar $G=(\Sigma, \PC, \triangleright, S)$ is
$L(G)=\set{a_1 \ldots a_n}{\text{there exists a string of categories } A_1; \ldots; A_n
\text{ with } a_i \triangleright A_i \text{ for all } i,
\text{ which is reducible to } S}$.

\end{definition}

In general, categorial grammars are equal in power to context-free grammars.

\begin{oldtheorem}[Bar-Hillel et al.~\cite{BarhillelGaifmanShamir}]\label{Th:BarHillel}
A language $L \subseteq \Sigma^+$ is defined by a categorial grammar
if and only if
it is defined by a context-free grammar.
\end{oldtheorem}

\subsection{Additional terminology}

Some extra terminology is used in the paper to facilitate arguments involving categories.
For every subformula $B \BS A$ or $A / B$ of a category,
the category $A$ is called a \emph{numerator}
and $B$ is a \emph{denominator}.
For every category $C$,
the sets of all categories occurring therein as numerators and as denominators
are denoted as follows.

\begin{definition}
For every category $C$, the set of its numerators $num(C)$
and the set if its denominators $den(A)$
are defined inductively on its structure.
\begin{itemize}
\item
	For a primitive category $C=p$,
	let $num(p) = \varnothing$ and $den(p) = \varnothing$.
\item
	Let $num(B \BS A) = num(A/B) = \{A\} \cup num(A)$
	and $den(B \BS A) = den(A/B) = \{B\} \cup den(A)$.
\end{itemize}
The definition is extended to a string of categories $A_1; \ldots; A_n$
as $num(A_1 \ldots A_n)=\bigcup_{i=1}^n num(A_i)$
and $den(A_1 \ldots A_n)=\bigcup_{i=1}^n den(A_i)$.
\end{definition}

One can state the following very simple but useful property.
\begin{proposition}
If a string of multiple categories
is reducible to a single category $A$,
then at least one of the categories in this string
has $A$ in the set of its numerators.
\end{proposition}

Indeed, if there are at least two categories in the string,
then the last step in a derivation of $A$
must collapse two categories to $A$,
and these could be either $A/B; B$, for some $B$, or $B; B \BS A$.

\subsection{The case of unique category assignment}

A categorial grammar $(\Sigma, \PC, \triangleright, S)$
is said to be \emph{with unique category assignment}
if for every symbol $a \in \Sigma$
there is a unique category $A$ with $a \triangleright A$.

\begin{example}\label{BCG_example_2}
The categorial grammar $(\Sigma, \PC, \triangleright, p)$,
where $\Sigma=\{a,b,c\}$, $\PC = \{ p, q \}$,
the category assignments are
$c \triangleright p$,
$a \triangleright q/p$ and
$b \triangleright q \BS p$,
and $p$ is the target category,
describes the language $\{a^n c b^n \mid n \geqslant 0 \}$.
\end{example}

In the above grammar, to put it plainly,
$c$ is always the center marker,
$a$ is always appended to a central substring from the left,
and $b$ is always appended from the right to a central substring extended with $a$.
But as long as any symbol may have more than one syntactic role,
categorial grammars with unique category assignment
can no longer express that;
e.g., the simplest possible context-free grammar $A \to aA \ | \ a$
has $a$ either in the role of a symbol pre-pended to a suffix,
or in the role of a last symbol in a string---
and the proposed subclass of categorial grammars
already cannot describe this language!

\begin{example}
The language $a^+$ is not representable
by any categorial grammar with unique category assignment.
\end{example}
\begin{proof}
Since $a$ is in the language,
the category assignment for $a$ must be $S$.
Then, the string of categories $S;S$ is
not reducible to any category,
and hence no other string is in the language.
\end{proof}

\section{Simulation of multiple category assignment}\label{section_simulation_multiple_assignment}

Let $G=(\Sigma, N, R, S)$ be a context-free grammar
in the Greibach normal form,
that is, with all rules of the form $X \to a$, $X \to a Y$ or $X \to a Y Z$,
with $X, Y, Z \in N$ and $a \in \Sigma$.

For every symbol $a \in \Sigma$,
consider all rules with $a$ on the right-hand side.
In the classical construction of an equivalent categorial grammar,
every such rule induces a category assigned to $a$:
$a \triangleright X$ for each rule $X \to a$,
$a \triangleright X / Y$ for each rule $X \to a Y$,
and $a \triangleright (X / Y) / Z$ for each $X \to a Y Z$.
This is an assignment of multiple categories to a single symbol.

An idealized goal would be to construct a single string of categories,
that would reduce to each of the desired categories.
Such a string would effectively replace $a$.
However, it will soon be demonstrated that this is impossible as stated:
e.g., no string of categories can be reducible to $X$ and to $X/Y$,
where $X$ and $Y$ are primitive categories.

Can a single string of categories
at all be reducible to two different categories?
In fact, yes! The following small example became the starting point of this research.

\begin{example}
Let $p$ be a category.
Then, the string of categories
$(p/p)/p; p; p; p\BS (p \BS p)$
reduces to $p/p$ and to $p \BS p$.
\begin{equation*}
		\infer{p \BS p}{
			\infer{p}{
				\infer{p/p}{
					(p/p)/p
					&
					p
				}
				&
				p
			}
			&
			p\BS (p \BS p)
		}
		\qquad
		\infer{p/p}{
			(p/p)/p
			&
			\infer{p}{
				p
				&
				\infer{p \BS p}{
					p
					&
					p\BS (p \BS p)
				}
			}
		}
\end{equation*}
\end{example}

However, there are substantial limitations to such multiple reductions.
It turns out that there are certain invariants,
which are preserved in the course of any reduction,
so that whenever a string of categories is reduced to multiple categories,
the resulting categories must be similar under some equivalence relation.

\begin{definition}
[van Benthem~{\cite[Ch.~6]{vanBenthem_language}}]
Let $p$ be a primitive category.
Then, the \emph{count of $p$} in any category $A$,
denoted by $\#_p(A)$,
is an integer defined inductively on the structure of $A$ as follows.
\begin{itemize}
\item
		$\#_p(p)=1$
\item
		$\#_p(q)=0$ for any primitive category $q \neq p$
\item
		$\#_p(A/B)=\#_p(B \BS A)=\#_p(A)-\#_p(B)$
\end{itemize}
The count of $p$ in a string of categories $A_1 \ldots A_\ell$
is $\#_p(A_1 \ldots A_\ell)=\sum_{i=1}^\ell \#_p(A_i)$.
\end{definition}

Reduction rules in categorial grammars
preserve the count of all primitive categories,
simply because $\#_p(A/B; B) = \#_p(A)-\#_p(B)+\#_p(B)=\#_a(A)$,
and similarly for $\#_p(B; B \BS A)$.

\begin{lemma}\label{count_invariant}
Let $u$ and $v$ be two strings of categories, such that $u$ reduces to $v$.
Then, $\#_{p}(u) = \#_{p}(v)$
for any primitive category $p$.
\end{lemma}

Hence, any two categories obtained by reducing a single string of categories
must have the same count of any primitive category.
In particular, there is no string of categories $u$
that would be reducible to two distinct primitive categories $p$ and $t$.

On the other hand, this invariant
is not an obstacle to constructing a string of categories
that would be reducible to $p/p$ and to $t/t$,
because both categories have zero count of $p$ and of $t$.
Such a string of categories is constructed below.

\begin{example}\label{example_x}
Let $p$ and $t$ be any categories.
Then, the string of categories
$(p/p)/(t/t); t/t; t; t \BS p; p \BS (t/t)$
is reducible to $p/p$ and to $t/t$.
\begin{equation*}
		\infer{p/p}{
			(p/p)/(t/t)
			&
			\infer{t/t}{
				\infer{p}{
					\infer{t}{
						t/t
						&
						t
					}
					&
					t \BS p
				}
				&
				p \BS (t/t)
			}
		}
		\qquad
		\infer{t/t}{
			\infer{p}{
				\infer{p/p}{
					(p/p)/(t/t)
					&
					t/t
				}
				&
				\infer{p}
				{
					t
					&
					t \BS p
				}
			}
			&
			p \BS (t/t)
		}
\end{equation*}
\end{example}

This example will serve as a model for the subsequent constructions in the general case.
The idea is to circumvent count-based invariants
by \emph{splitting} each primitive category $p$ into $p/p$.
For instance, this would change a category $p/(s \BS r)$
into $(p/p)/((s/s) \BS (r/r))$.
The following notation is introduced.

\begin{definition}
The function
$\phi$ mapping categories to categories
is defined inductively on the structure of the argument category:
\begin{itemize}
\item
		if $p \in Pr$ then $\phi(p) = p / p$,
\item
		if $C = A / B$ then $\phi(C) = \phi(A) / \phi(B)$,
\item
		if $C = B \BS A$ then $\phi(C) = \phi(B) \BS \phi(A)$.
\end{itemize}
\end{definition}

Now the original goal of expressing the choice between multiple categories
can be restated as follows:
for any given categories $A_1, \ldots, A_n$,
construct a string of categories that would reduce
exactly to the categories $\phi(A_1), \ldots, \phi(A_n)$.
It is sufficient to do this for categories
of three forms $p$, $p / q$ and $(p / q) / s$,
because these are the categories obtained in the simulation
of a context-free grammar in the Greibach normal form.

The construction is first presented for the case of two categories.

\begin{lemma}\label{z_A_B_lemma}
Let $A$ and $B$ be two distinct categories
of the form $p$, $p / q$ or $(p / q) / s$.
Then there exists a string of categories $z_{A,B}$, 
such that
\begin{enumerate}\renewcommand{\theenumi}{\Roman{enumi}}
\item	\label{z_A_B_lemma__reduces}
		$z_{A,B}$ reduces to $\phi(A)$ and to $\phi(B)$;
\item	\label{z_A_B_lemma__prefix_suffix}
		no proper suffix of $z_{A,B}$ is reducible to $\phi(A)$,
		and no proper prefix of $z_{A,B}$ is reducible to $\phi(B)$.
\end{enumerate}
\end{lemma}

The proof of the lemma is structured as follows.
First, the construction of $z_{A,B}$ is presented,
along with a proof of the first part.
Later, the second part will be proved separately.

\begin{proof}[Proof of Part~\ref{z_A_B_lemma__reduces}] 
The desired string of categories $z_{A,B}$ is constructed in several stages, starting from shorter strings with useful properties.

The first to be defined is a string $x_{A,t}$
(where $t$ is a new primitive category, that is, one that does not occur in the category $A$),
which should reduce both to $\phi(A)$ and $t/t$.
Since $A$ can be of three possible forms,
$x_{A,t}$ is defined differently for each of these forms.
\begin{align*}
		x_{p,t} &=
		(p/p)/(t/t); t/t; t; t \BS p; p \BS (t/t) \\
		x_{p/q, \: t} &=
		((p/p)/(q/q))/(t/t); t/t; t; t \BS (q/q); q ; q\BS p; p \BS (t/t) \\
		x_{(p/q) /s, \: t} &=
		(((p/p)/(q/q))/(s/s))/(t/t); t/t; t; t \BS (s/s);  
          s; s\BS (q/q); q; q\BS p; p \BS (t/t)
\end{align*}
	
\begin{ourclaim}\label{x_reductions}
For every category $A$, the string $x_{A,t}$
is reducible to both $\phi(A)$ and $\phi(t)$.
\end{ourclaim}
\begin{proof}
The reductions for the string $x_{p,t}$ are exactly the same as in Example~\ref{example_x}.

The following are reductions for $x_{p/q, \: t}$. 
\begin{equation*}
	\infer{(p/p)/(q/q)}{
		((p/p)/(q/q))/(t/t)
		&
		\infer{t/t}{
			\infer{p}{
				\infer{q}{
					\infer{q/q}{
						\infer{t}{
							t/t
							&
							t
						}
						&
						t\BS (q/q)
					}
					&
					q
				}
				&
				q \BS p
			}
			&
			p \BS (t/t)
		}
	}
\end{equation*}

\begin{equation*}
	\infer{t/t}{
		\infer{p}{
			\infer{p/p}{
				\infer{(p/p)/(q/q)}{
					((p/p)/(q/q))/(t/t)
					&
					t/t
				}
				&
				\infer{q/q}{
					t
					&
					t\BS (q/q)
				}
			}
			&
			\infer{p}{
				q
				&
				q\BS p
			}
		}
		&
		p\BS(t/t)
	}
\end{equation*}

The reductions for $x_{(p/q)/s, \: t}$ are as follows.

\begin{equation*}
	\infer{((p/p)/(q/q))/(s/s)}{
		(((p/p)/(q/q))/(s/s))/(t/t)
		&
		\infer{t/t}{
			\infer{p}{
				\infer{q}{
					\infer{q/q}{
						\infer{s}{
							\infer{s/s}{
								\infer{t}{
									t/t
									&
									t
								}
								&
								t\BS (s/s)
							}
							&
							s
						}
						&
						s\BS (q/q)
					}
					&
					q
				}
				&
				q\BS p
			}
			&
			p\BS (t/t)
		}
	}
\end{equation*}

\begin{equation*}
	\infer{t/t}{
		\infer{p}{
			\infer{p/p}{
				\infer{(p/p)/(q/q)}{
					\infer{((p/p)/(q/q))/(s/s)}{
						(((p/p)/(q/q))/(s/s))/(t/t)
						&
						t/t
					}
					&
					\infer{s/s}{
						t
						&
						t\BS (s/s)
					}
				}
				&
				\infer{q/q}{
					s
					&
					s\BS (q/q)
				}
			}
			&
			\infer{p}{
				q
				&
				q\BS p
			}
		}
		&
		p\BS (t/t)
	}
\end{equation*}

\end{proof}
Next, another string $y_{t,B}$ is defined,
which is reducible to $t/t$ and $\phi(B)$.
Although this property is exactly the same as for $x_{A,t}$,
the string itself is quite different from $x_{A,t}$,
and its further properties stated in
Part~\ref{z_A_B_lemma__prefix_suffix} of this lemma
will actually be symmetric to those of $x_{A,t}$.
For the time being, $y_{t,B}$ is defined in three cases as follows:
\begin{align*}
y_{t,p} &=
(t/t)/(p/p); p/p; p; p \BS t; t \BS (p/p) \\
y_{t, \: p/q} &=
(t/t)/((p/p)/(q/q)); (p/p)/(q/q); q/q; q; q\BS p; p \BS t; t \BS ((p/p)/(q/q)) \\
y_{t, \: (p/q)/s} &=
(t/t)/(((p/p)/(q/q))/(s/s)); ((p/p)/(q/q))/(s/s); s/s; s; \\
&\hspace*{40mm}
s \BS (q/q); q; q \BS p; p \BS t; t \BS (((p/p)/(q/q))/(s/s))
\end{align*}
This string reduces to the same categories
as the earlier $x_{A,t}$. 

\begin{ourclaim}\label{y_reductions}
For every category $B$, the string $y_{t,B}$ is
reducible to both $\phi(t)$ and $\phi(B)$.
\end{ourclaim}
\begin{proof}
The following are reductions for $y_{t,p}$. 
\begin{equation*}
	\infer{t/t}{
		(t/t)/(p/p)
		&
		\infer{p/p}{
			\infer{t}{
				\infer{p}{
					p/p
					&
					p
				}
				&
				p\BS t
			}
			&
			t\BS (p/p)
		}
	}
\end{equation*}

\begin{equation*}
	\infer{p/p}{
		\infer{t}{
			\infer{t/t}{
				(t/t)/(p/p)
				&
				p/p
			}
			&
			\infer{t}{
				p
				&
				p\BS t
			}
		}
		&
		t\BS (p/p)
	}
\end{equation*}

The following are reductions for $y_{t, \: p/q}$. 

\begin{equation*}
	\infer{t/t}{
		(t/t)/((p/p)/(q/q))
		&
		\infer{(p/p)/(q/q)}{
			\infer{t}{
				\infer{p}{
					\infer{p/p}{
						(p/p)/(q/q)
						&
						q/q
					}
					&
					\infer{p}{
						q
						&
						q\BS p
					}
				}
				&
				p\BS t
			}
			&
			t\BS ((p/p)/(q/q))
		}
	}
\end{equation*}

\begin{equation*}
	\infer{(p/p)/(q/q)}{
		\infer{t}{
			\infer{t/t}{
				(t/t)/((p/p)/(q/q))
				&
				(p/p)/(q/q)
			}
			&
			\infer{t}{
				\infer{p}{
					\infer{q}{
						q/q
						&
						q
					}
					&
					q\BS p
				}
				&
				p\BS t
			}
		}
		&
		t\BS ((p/p)/(q/q))
	}
\end{equation*}

The string $y_{t,\: (p/q)/s}$ is too long to be reduced in one take.
Consider first the reduction of the following substring.
\begin{equation*}
	u = ((p/p)/(q/q))/(s/s); s/s; s; s \BS (q/q); q; q \BS p; p \BS t
\end{equation*}
Then, the string $y_{t,\: (p/q)/s}$ is of the following form. 
$$(t/t)/(((p/p)/(q/q))/(s/s));u;t\BS (((p/p)/(q/q))/(s/s))$$
The substring $u$ is reducible to the string $((p/p)/(q/q))/(s/s);t$ and to the category $t$. 
\begin{equation*}
	\infer{t}{
		\infer{p}{
			\infer{p/p}{
				\infer{(p/p)/(q/q)}{
					((p/p)/(q/q))/(s/s)
					&
					s/s
				}
				&
				\infer{q/q}{
					s
					&
					s\BS (q/q)
				}
			}
			&
			\infer{p}{
				q
				&
				q\BS p
			}
		}
		&
		p\BS t
	}
\end{equation*}

\begin{equation*}
	((p/p)/(q/q))/(s/s) \quad
	\infer{t}{
		\infer{p}{
			\infer{q}{
				\infer{q/q}{
					\infer{s}{
						s/s
						&
						s
					}
					&
					s\BS (q/q)
				}
				&
				q
			}
			&
			q\BS p
		}
		&
		p\BS t
	}
\end{equation*}

Hence, $y_{t,(p/q)/s}$ is reducible to the following strings.
\begin{align*}
	& (t/t)/(((p/p)/(q/q))/(s/s));t;t\BS (((p/p)/(q/q))/(s/s)) \\
	& (t/t)/(((p/p)/(q/q))/(s/s));(((p/p)/(q/q))/(s/s);t;t\BS (((p/p)/(q/q))/(s/s))
\end{align*}
The former string is reducible to $t/t$,
and the latter to $((p/p)/(q/q))/(s/s)$.
\end{proof}

The desired string of categories $z_{A,B}$ should combine $x_{A,t}$ and $y_{t,B}$ in order to be reducible both to $\phi(A)$ and $\phi(B)$.

Let $k_1$, $k_2$, $k_3$ and $k_4$ be four new primitive categories,
and define the desired string of categories $z_{A,B}$ as follows.
\begin{multline*}
	z_{A,B}
	=
	\phi(A) / k_1; k_1/ (t/t); x_{A,t} ; \phi(A) \BS k_2 ;
	k_2 \BS ((t/t) / k_3) ; k_3 / \phi(B) ; y_{t,B} ;
	\\ (t/t) \BS k_4 ; k_4 \BS \phi(B)
\end{multline*}

Denote the following substrings of $z_{A,B}$ by $u$ and $v$
(which not to be confused with the unrelated strings $u,v$
in the proof of Claim~\ref{y_reductions}).
\begin{align*}
	u & = \phi(A) / k_1; k_1/ (t/t); x_{A,t} ; \phi(A) \BS k_2 ;
	k_2 \BS ((t/t) / k_3) ; k_3 / \phi(B) ; y_{t,B} \\
	v & = x_{A,t} ; \phi(A) \BS k_2 ;
	k_2 \BS ((t/t) / k_3) ; k_3 / \phi(B) ; y_{t,B} ; (t/t) \BS k_4 ; k_4 \BS \phi(B)
\end{align*}

Then, the string $z_{A,B}$ can be represented in the following two forms.
\begin{align*}
	z_{A,B} &= v;(t/t) \BS k_4 ; k_4 \BS \phi(B) \\
	z_{A,B} &= \phi(A) / k_1; k_1/ (t/t) ; u
\end{align*}
By the following reductions, both strings $u$ and $v$ are reducible to $t/t$. 
\begin{equation*}
	\infer{t/t}{
		\infer{(t/t)/k_3}{
			\infer{k_2}{
				\infer{\phi(A)}{x_{A,t}}
				&
				\phi(A)\BS k_2
			}
			&
			k_2\BS((t/t)/k_3)
		}
		&
		\infer{k_3}{
			k_3/\phi(B)
			&
			\infer{\phi(B)}{
				\infer{k_4}{
					\infer{t/t}{y_{t,B}}
					&
					(t/t)\BS k_4
				}
				&
				k_4 \BS \phi(B)
			}
		}
	}	
\end{equation*}

\begin{equation*}
	\infer{t/t}{
		\infer{(t/t)/k_3}{
			\infer{k_2}{
				\infer{\phi(A)}{
					\phi(A)/k_1
					&
					\infer{k_1}{
						k_1/(t/t)
						&
						\infer{t/t}{
							x_{A,t}
						}
					}
				}
				&
				\phi(A) \BS k_2
			}
			&
			k_2 \BS ((t/t)/k_3)
		}
		&
		\infer{k_3}{
			k_3/ \phi(B)
			&
			\infer{\phi(B)}{y_{t,B}}
		}
	}
\end{equation*}

Hence, the string $z_{A,B}$ is reducible to $t/t; (t/t) \BS k_4 ; k_4 \BS \phi(B)$,
which is in turn reducible to $\phi(B)$. 
At the same time,
the string $z_{A,B}$ is reducible to $\phi(A) / k_1; k_1/ (t/t) ; t/t$,
which is reducible to $\phi(A)$.

The string $z_{A,B}$ can also be reduced to $t/t$;
but this does not affect the statement of the lemma,
and is not needed in the following constructions.
\end{proof}

The proof of the second condition in Lemma~\ref{z_A_B_lemma}
is based on the following technical result,
which will also be reused later in other proofs.

\begin{proposition}\label{main_technical_proposition}
Let $A$, $B$, $C$ and $D$ be any categories,
and let $\alpha$, $\beta$ and $\gamma$ be any strings of categories. 
Let $w = \alpha;A / C; \beta ; D \BS B; \gamma$,
and assume that $A / C \notin den(w)$,
$D \BS B \notin den(w)$, 
no proper suffix of $\beta$ is reducible to $D$,
and no proper prefix of $\beta$ is reducible to $C$.
Then, if $w$ is reducible to a single category $E$,
then at least one of the strings $\alpha;A / C; C ; D \BS B; \gamma$
and $\alpha;A / C; D ; D \BS B; \gamma$
is also reducible to $E$.
\end{proposition}
\begin{proof}
Consider the step in the reduction of $w$ to $E$,
when one of the categories $A/C$ or $D \BS B$
is first involved in a reduction.
Since neither category occur in the denominators,
either $A/C$ is multiplied by $C$,
or $D$ is multiplied by $D \BS B$.

If this is a reduction of $A/C; C$ to $A$,
then $C$ is obtained from a prefix of $\beta ; D \BS B; \gamma$.
Since the category $D \BS B$ has not been involved in any reductions yet,
this $C$ must have been obtained from a prefix of $\beta$.
This cannot be a proper prefix of $\beta$ by the assumption,
and hence the entire $\beta$ must have been reduced to $C$.
If any of the early derivation steps affect $\gamma$,
they can be postponed until the reduction of $\beta$ to $C$ is complete.
The resulting reduction begins by transforming $\alpha;A / C; \beta ; D \BS B; \gamma$
to $\alpha;A / C; C ; D \BS B; \gamma$,
and then eventually reduces it to $E$, as claimed.

The case of a reduction of $D; D \BS B$ to $B$,
where $D$ has been obtained by reducing some suffix of $\alpha;A / C; \beta$,
is symmetric:
$A/C$ remains intact, hence $D$ is obtained from a suffix of $\beta$,
which cannot be a proper suffix by the assumption.
Therefore, $\beta$ is reducible to $D$,
and the reduction of $\alpha;A / C; \beta ; D \BS B; \gamma$ to $E$
goes through the intermediate string of categories $\alpha;A / C; D ; D \BS B; \gamma$.
\end{proof}

\begin{proof}[Proof of Part~\ref{z_A_B_lemma__prefix_suffix} of Lemma~\ref{z_A_B_lemma}]
It has to be proved that proper prefixes of $z_{A,B}$ cannot be reduced to $\phi(B)$
and proper suffixes of $z_{A,B}$ are not reducible to $\phi(A)$.
The proof is a case study of all proper prefixes and suffixes.
For those prefixes and suffixes that do not split $x_{A,t}$ nor $y_{t,B}$
(that is, contain either the whole substring, or no part thereof),
the desired property
shall be obtained using Proposition~\ref{main_technical_proposition}. 
And if any of these substrings are split,
then the proof relies on the following properties of their prefixes and suffixes,
which will be established first.

\begin{ourclaim}\label{x_A_t}
For every category $A$, 
no proper suffix of $x_{A,t}$ is reducible to $\phi(A)$
and no proper prefix of $x_{A,t}$ is reducible to $(t/t)$.
\end{ourclaim}
\begin{proof}
The string $x_{A,t}$ has only one category with $t/t$ in its set of numerators,
and this is the last category in the string. 
Hence, no proper prefix is reducible to $t/t$.

At the same time, in the string $x_{A,t}$,
the category $\phi(A)$ occurs in the set of numerators only for the first category.
No proper suffix can then be reduced to $\phi(A)$.
\end{proof}

\begin{ourclaim}\label{y_t_B_claim}
For every category $B$,
no proper suffix of $y_{t,B}$ is reducible to $(t/t)$
and no proper prefix of $y_{t,B}$ is reducible to $\phi(B)$.
\end{ourclaim}
\begin{proof}
The string $y_{t,B}$ has only one category with $\phi(B)$ in the numerator,
and this is the last category in the string.
Hence, no proper prefix is reducible to $\phi(B)$. 

Similarly, there is only one category in $y_{t,B}$ with $t/t$ in the numerator,
it is the the first category in the string,
and so no proper suffix can be reduced to $t/t$. 
\end{proof}

The proof is given separately for the cases of a suffix and of a prefix.

\textit{The case of a suffix:}	
First consider any proper suffix of $z_{A,B}$
that begins not at the borders of $x_{A,t}$ and of $y_{t,B}$
and not inside either of these substrings.
There are four possible starting points of such suffixes.
\begin{equation*}
	\phi(A) / k_1
	\stackrel{\downarrow}{;}
	k_1/ (t/t);
	x_{A,t};
	\phi(A) \BS k_2
	\stackrel{\downarrow}{;}
	k_2 \BS ((t/t) / k_3)
	\stackrel{\downarrow}{;}
	k_3 / \phi(B);
	y_{t,B};
	(t/t) \BS k_4
	\stackrel{\downarrow}{;}
	k_4 \BS \phi(B)
\end{equation*}
On both sides of each of these four arrows
there are categories involving a new basic category
($k_1$, $k_2$, $k_3$ and $k_4$, respectively).
If they are separated, then the suffix containing only one of them
cannot be reduced to $\phi(A)$,
which makes these cases impossible.

Another possibility is that a suffix
begins right after $x_{A,t}$ or $y_{t,B}$.
\begin{equation*}
	\phi(A) / k_1
	;
	k_1/ (t/t);
	x_{A,t}
	\stackrel{\downarrow}{;}
	\phi(A) \BS k_2
	;
	k_2 \BS ((t/t) / k_3)
	;
	k_3 / \phi(B);
	y_{t,B}
	\stackrel{\downarrow}{;}
	(t/t) \BS k_4
	;
	k_4 \BS \phi(B)
\end{equation*}
In either case, the first category in the suffix
has division to the left, and hence could participate only in a reduction
with some larger category on the right.
However, this first category
does not occur in any denominator,
and hence no other category can be reduced by it.

The next possibility is that the suffix begins 
\textbf{inside $x_{A,t}$}.
Suppose the whole thing is reduced to one category.
Then, a single category $k_2$ must be obtained at some step of the reduction.
But to obtain $k_2$, the chosen suffix of $x_{A,t}$ need to be reduced to $\phi(A)$, which is impossible by Claim~\ref{x_A_t}.

Let the suffix begin \textbf{immediately before $x_{A,t}$}.
Like in the previous case,
$k_2$ must be obtained,
and then the only option is to reduce $x_{A,t}$ to $\phi(A)$,
which is possible this time.
Then, the substring $x_{A,t} ; \phi(A) \BS k_2 ; k_2 \BS ((t/t) / k_3)$
must be reduced to $(t/t) / k_3$.
Then the following string must be reduced to $\phi(A)$.
\begin{equation*}
	(t/t) / k_3;
	k_3 / \phi(B);
	y_{t,B};
	(t/t) \BS k_4
	;
	k_4 \BS \phi(B)
\end{equation*}
Even though $\phi(A)$ is not explicitly written in the above string of categories,
the possibility of reducing this string to $\phi(A)$ cannot be ruled out right away,
because if $A$ is a subterm of $B$,
then, potentially, one could obtain $\phi(A)$
by first reducing a substring to $\phi(B)$,
and then applying further reductions to obtain $\phi(A)$.
The argument uses Proposition~\ref{main_technical_proposition} and the properties of $y_{t,B}$. 
By Claim~\ref{y_t_B_claim},
no proper suffix of $y_{t,B}$ is reducible to $t/t$,
and no proper prefix of $y_{t,B}$ is reducible to $\phi(B)$.
Then, by Proposition~\ref{main_technical_proposition} for
$A/C = k_3/\phi(B)$, $\beta = y_{t,B}$ and $D \BS B = (t/t)\BS k_4$,
if this string is reducible to $\phi(A)$,
then one of the following two strings must be reducible to $\phi(A)$:
\begin{equation*}
	(t/t) / k_3;
	k_3 / \phi(B);
	\phi(B);
	(t/t) \BS k_4
	;
	k_4 \BS \phi(B)
\end{equation*}
\begin{equation*}
	(t/t) / k_3;
	k_3 / \phi(B);
	t/t;
	(t/t) \BS k_4
	;
	k_4 \BS \phi(B)
\end{equation*}
However, the former string is reducible only to $\phi(B)$,
and the latter only to $t/t$.

If the suffix begins \textbf{immediately before $y_{t,B}$}
and is reducible to $\phi(A)$,
then $k_4$ must be obtained in the course of the reduction,
and accordingly the substring $y_{t,B}$ must be reduced to $t/t$.
Next, $t/t$ and $(t/t) \BS k_4$ must be reduced to $k_4$,
and, at last, $k_4$ and $k_4 \BS \phi(B)$ have to be reduced to $\phi(B)$.
So the only possible category obtained in such a reduction is $\phi(B)$.

Finally, let the suffix begin \textbf{inside $y_{t,B}$},
and suppose the whole thing is reduced to one category.
Then, $k_4$ must be obtained, and the chosen suffix of $y_{t,B}$
need to be reduced to $t/t$, which is impossible. 
This concludes the proof in the case of a suffix.

\textit{The case of a prefix:}
The proof is similar, but not entirely symmetric to the case of a suffix.

Again, the chosen prefix cannot end at the borders of $x_{A,t}$ and $y_{t,B}$, 
because this would separate new primitive categories. 

Moreover, the chosen prefix cannot end right before $x_{A,t}$ and $y_{t,B}$, 
because the last categories of such strings cannot participate in any further reduction. 
Furthermore, the chosen prefix cannot end inside of $x_{A,t}$ and $y_{t,B}$, 
also because the last categories of such strings cannot be reduced.

Thus, the prefix can end either right after $x_{A,t}$, or right after $y_{t,B}$. 

The first case is ruled out right away,
because for the category $k_1/(t/t)$ to be reduced,
the string $x_{A,t}$ has to be reduced to $t/t$.
Hence, such a suffix is reducible only to $\phi(A)$. 

Now let the suffix end right after $y_{t,B}$.
Hence, $y_{t,B}$ has to be reduced to $\phi(B)$ to obtain $k_3$.
Then the following string must be reduced to $\phi(B)$.
\begin{equation*}
	\phi(A) / k_1 ; k_1/ (t/t);x_{A,t};
	\phi(A) \BS k_2;k_2 \BS ((t/t) / k_3);k_3 
\end{equation*}
By Claim~\ref{x_A_t}, no proper prefix of 
$x_{A,t}$ is reducible to $t/t$, and no proper
suffix of $x_{t,B}$ is reducible to $\phi(A)$.
Then, by
Proposition~\ref{main_technical_proposition} for
$A/C = k_1/ (t/t)$, $\beta =x_{A,t}$ and $C\BS D = \phi(A) \BS k_2$,
one of the following strings must be reducible to $\phi(B)$:
\begin{align*}
	\phi(A) / k_1;k_1/ (t/t);t/t;\phi(A) \BS k_2;
	k_2 \BS ((t/t) / k_3);k_3 \\
	\phi(A) / k_1;k_1/ (t/t);\phi(A);
	\phi(A) \BS k_2;k_2 \BS ((t/t) / k_3);k_3 
\end{align*}
This, however, is not the case:
the former string is reducible only to $t/t$,
and the latter only to $\phi(A)$.
\end{proof}

A substring $z_{A,B}$
reducible to $\phi(A)$ and $\phi(B)$
simulates the assignment of multiple categories:
now the general idea is to use this substring
instead of a symbol,
whenever two categories $\{A,B\}$ need to be assigned to that symbol.
Thus, a string of original symbols would be replaced
by a string of substrings of the form $z_{A,B}$.
However, the properties of $z_{A,B}$ do not guarantee
that there would be no interaction
between neighbouring substrings $z_{A,B}$ and $z_{C,D}$,
so that $y_{t,B}$ from the first substring
would get contaminated
with $x_{C,t}$ from the second substring,
with unpredictable results.

In order to prevent this, the strings $z_{A,B}$
will be extended by wrapping them with a number of new categories
this impose some limitations on the possible order of reductions,
and eventually eliminate the danger of contamination.
The first step is to apply division
by new \emph{sentinel categories} $\ell$ and $r$
on both sides of each string representing two categories $\{A,B\}$.
This will be a string $u_{A,B}$
that reduces to $\ell \BS (\phi(A)/r)$ 
and $\ell \BS (\phi(B)/r)$.

Denote by $\psi(A)$ the category $\phi(A)$ divided by $r$ on the right
and then by $\ell$ on the left.
\begin{definition}
For every category $A$,
let $\psi(A)=\ell \BS (\phi(A)/r)$.
\end{definition}

The next point of the plan
is to replace the $\phi$-encoding of a category
with the new $\psi$-encoding.
It actually has the same basic properties
as earlier established for the $\phi$-encoding.

\begin{lemma}[cf.~Lemma~\ref{z_A_B_lemma}]\label{u_A_B_lemma}
Let $A$ and $B$ be two categories
of the form $p$, $p / q$ or $(p / q) / s$.
Then:
\begin{enumerate}\renewcommand{\theenumi}{\Roman{enumi}}
\item
\label{u_A_B_lemma__reduces}
there exists a string of categories $u_{A,B}$
that reduces to $\psi(A)$ and to $\psi(B)$,
and, furthermore,
\item
\label{u_A_B_lemma__prefix_suffix}
no proper suffix of $u_{A,B}$ is reducible to $\psi(A)$
and no proper prefix of $u_{A,B}$ is reducible to $\psi(B)$.
\end{enumerate}
\end{lemma}
\begin{proof}
The construction uses four new categories $o_1$, $o_2$, $o_3$, $o_4$, and is done in two stages.
First, division by $r$ is applied on the right,
resulting in the following intermediate string,
and division by $\ell$ will be applied later.
\begin{multline*}
    z'_{A,B}
    =
    (\phi(A)/r)/(\phi(B)/r); \phi(A)/o_1; o_1/\phi(B); z_{A,B};
    \\
    \phi(A)\BS o_2;
    o_2 \BS (\phi(B)/r); r; \phi(A) \BS (\phi(B)/r) 
\end{multline*}

\begin{ourclaim}\label{z'_reductions}
For every two distinct categories $A$ and $B$,
the string $z'_{A,B}$
is reducible to both $\phi(A)/r$ and $\phi(B)/r$.
\end{ourclaim}
\begin{proof}
	Denote by $u$ and $v$ the following 
	two substrings of the string $z'_{A,B}$. 
	\begin{align*}
		u &= \phi(A)/o_1; o_1/\phi(B); z_{A,B};
		\phi(A)\BS o_2;
		o_2 \BS (\phi(B)/r); r; \phi(A) \BS (\phi(B)/r) \\
		v &= (\phi(A)/r)/(\phi(B)/r); \phi(A)/o_1; o_1/\phi(B); z_{A,B};
		\phi(A)\BS o_2;
		o_2 \BS (\phi(B)/r); r
	\end{align*}
	Thus, the string $z'_{A,B}$ is represented as
	$z'_{A,B}= (\phi(A)/r)/(\phi(B)/r);u $,
	and at the same time as
	$z'_{A,B}= v;\phi(A) \BS (\phi(B)/r)$. 
	The following reductions show that $u$ is reducible to $\phi(B) / r$,
	and $v$ is reducible to $\phi(A)$. 
	\begin{equation*}
		\infer{\phi(B)/ r}{
			\infer{\phi(A)}{
				\phi(A)/o_1
				&
				\infer{o_1}{
					o_1/\phi(B)
					&
					\infer{\phi(B)}{
						\infer{\phi(B)/r}{
							\infer{o_2}{
								\infer{\phi(A)}{
									z_{A,B}
								}
								&
								\phi(A)\BS o_2
							}
							&
							o_2 \BS (\phi(B)/r)
						}
						&
						r
					}
				}
			}
			&
			\phi(A) \BS (\phi(B)/r)
		}
	\end{equation*}

	\begin{equation*}
		\infer{\phi(A)}{
			\infer{\phi(A)/r}{
				(\phi(A)/r)/(\phi(B)/r)
				&
				\infer{\phi(B)/r}{
					\infer{o_2}{
						\infer{\phi(A)}{
							\phi(A)/o_1
							&
							\infer{o_1}{
								o_1/\phi(B)
								&
								\infer{\phi(B)}{
									z_{A,B}
								}
							}
						}
						&
						\phi(A)\BS o_2 
					}
					&
					o_2 \BS (\phi(B)/r)
				}
			}
			&
			r
		}
	\end{equation*}
	Hence, $z'_{A,B}$ is reducible to $\phi(A)/r$ and $\phi(B)/r$.
\end{proof}

\begin{ourclaim}\label{z'_A_B_claim}
For every two distinct categories $A$ and $B$, 
no proper suffix of $z'_{A,B}$ is reducible to $\phi(A)/r$,
and no proper prefix of $z'_{A,B}$ is reducible to $\phi(B)/r$.
\end{ourclaim}
\begin{proof}
	\textbf{The case of a suffix.}
	The desired category $\phi(A)/r$ occurs in a numerator
	only in the first category of the string $z'_{A,B}$. 
	Hence, no proper
	suffix of $z'_{A,B}$ can be reduced to
	$\phi(A)/r$.
	
	\textbf{The case of a prefix.}
	The category $\phi(B)/r$ occurs in $z'_{A,B}$ in a numerator twice:
	in $o_2 \BS (\phi(B)/r)$,
	and in $\phi(A) \BS (\phi(B)/r)$.
	If a proper prefix is reducible to $\phi(B)/r$,
	then at least it must contain $o_2 \BS (\phi(B)/r)$.
	There are two options: the chosen prefix might end
	with $r$, or end short of $r$.
	
	If the chosen prefix ends with $r$,
	then it cannot be reduced to $\phi(B)/r$,
	as the count of $r$ in this prefix can be observed to be zero, 
	whereas the count of $r$ in $\phi(B)/r$ is $-1$.
	
	Then, the only possibility for the chosen prefix
	is to end with $o_2 \BS (\phi(B)/r)$.
	Since no proper prefix of $x$ is reducible to $\phi(B)$
	and no proper suffix of $x$ is reducible to
	$\phi(A)$ due to Lemma~\ref{z_A_B_lemma},
	we can apply Proposition~\ref{main_technical_proposition}
	with $A/C = o_1 / \phi(B)$, $\beta = z_{A,B}$ and
	$D \BS B = \phi(A) \BS o_2$,
	and it asserts that,
	for the chosen suffix to be reducible to
	$\phi(B)/r$, one of the following strings have to
	be reducible to $\phi(B)/r$.
	\begin{align*}
		(\phi(A)/r)/(\phi(B)/r); 
		\phi(A)/o_1; o_1/\phi(B); 
		\phi(A); 
		\phi(A)\BS o_2;
		o_2 \BS (\phi(B)/r)
		\\
		(\phi(A)/r)/(\phi(B)/r); 
		\phi(A)/o_1; o_1/\phi(B); 
		\phi(B); 
		\phi(A)\BS o_2;
		o_2 \BS (\phi(B)/r)
	\end{align*}
	Each of these strings has only one reduction tree,
	so the former is not reducible to any single category,
	and the latter is reducible only to $\phi(A)/r$.
\end{proof}

Resuming the proof of Lemma~\ref{u_A_B_lemma},
the desired string of categories $u_{A,B}$ is obtained
from the intermediate string $z'_{A,B}$ defined above
by a very similar transformation that divides by $\ell$ on the left.
\begin{multline*}
    u_{A,B}
    =
    (\ell \BS (\phi(A)/r)) / (\phi(B)/r) ; \ell; (\ell \BS (\phi(A)/r)) / o_3;
    o_3 / (\phi(B)/r); z'_{A,B} ; \\
    (\phi(A)/r) \BS o_4 ; o_4 \BS (\phi(B)/r);
    (\ell \BS (\phi(A)/r)) \BS (\ell \BS (\phi(B)/r))
\end{multline*}

To prove Part~\ref{u_A_B_lemma__reduces} of Lemma~\ref{u_A_B_lemma},
denote by $u$ and $v$ the following two substrings of the string $u_{A,B}$. 
\begin{align*}
	u &= (\ell \BS (\phi(A)/r)) / o_3; o_3 / (\phi(B)/r); 
	z'_{A,B}; (\phi(A)/r) \BS o_4 ; o_4 \BS (\phi(B)/r) \\
	v &= \ell; (\ell \BS (\phi(A)/r)) / o_3; o_3 / (\phi(B)/r); 
	z'_{A,B}; (\phi(A)/r) \BS o_4 ; o_4 \BS (\phi(B)/r)
\end{align*}

Thus, the whole string $u_{A,B}$ is represented as
$(\ell \BS (\phi(A)/r)) / (\phi(B)/r) ; \ell; u; (\ell \BS (\phi(A)/r)) \BS (\ell \BS (\phi(B)/r))$,
and at the same time as
$(\ell \BS (\phi(A)/r)) / (\phi(B)/r) ; v; (\ell \BS (\phi(A)/r)) \BS (\ell \BS (\phi(B)/r))$. 
The following reductions show that $u$ is reducible to $\psi(A)$,
and $v$ is reducible to $\phi(B)/r$. 

\begin{equation*}
	\infer{\ell \BS (\phi(A)/r)}{
		(\ell \BS (\phi(A)/r)) / o_3
		&
		\infer{o_3}{
			o_3 / (\phi(B)/r)
			&
			\infer{\phi(B)/r}{
				\infer{o_4}{
					\infer{\phi(A)/r}{z'_{A,B}}
					&
					(\phi(A)/r) \BS o_4
				}
				&
				o_4 \BS (\phi(B)/r)
			}
		}
	}
\end{equation*}

\begin{equation*}
	\infer{\phi(B)/r}{
		\infer{o_4}{
			\infer{\phi(A)/r}{
				\ell
				&
				\infer{\ell \BS (\phi(A)/r)}{
					(\ell \BS (\phi(A)/r)) / o_3
					&
					\infer{o_3}{
						o_3 / (\phi(B)/r)
						&
						\infer{\phi(B)/r}{z'_{A,B}}
					}
				}
			}
			&
			(\phi(A)/r) \BS o_4
		}
		&
		o_4 \BS (\phi(B)/r)
	}
\end{equation*}

At last, the following reductions show that $u$ is 
reducible to $\psi(A)$ and $\psi(B)$.

\begin{equation*}
	\infer{\ell \BS (\phi(A)/r) }{
		(\ell \BS (\phi(A)/r)) / (\phi(B)/r)
		&
		\infer{\phi(B)/r}{
			\ell
			&
			\infer{\ell \BS (\phi(B)/r)}{
				\infer{\ell \BS (\phi(A)/r)}{
					v
				}
				&
				(\ell \BS (\phi(A)/r)) \BS (\ell \BS (\phi(B)/r))
			}
		}
	}
\end{equation*}

\begin{equation*}
	\infer{\ell \BS (\phi(B)/r)}{
		\infer{\ell \BS (\phi(A)/r)}{
			(\ell \BS (\phi(A)/r)) / (\phi(B)/r)
			&
			\infer{\phi(B)/r}{
				v
			}
		}
		&
		(\ell \BS (\phi(A)/r)) \BS (\ell \BS (\phi(B)/r))
	}
\end{equation*}
 
This completes the proof of Part~\ref{u_A_B_lemma__reduces}.
of Lemma~\ref{u_A_B_lemma},

Turning to Part~\ref{u_A_B_lemma__prefix_suffix} of Lemma~\ref{u_A_B_lemma},
in the case of a \textbf{prefix},
the desired category $\ell\BS(\phi(B)/r)$ occurs in a numerator
only in the last category of the string $u_{A,B}$.
Hence, no proper prefix of
$u_{A,B}$ can be reduced to $\ell\BS(\phi(B)/r)$.

In the case of a \textbf{suffix},
the category $\ell\BS(\phi(A)/r)$
occurs in a numerator only in the categories
$(\ell\BS(\phi(A)/r))/(\phi(B)/r)$ and 
$(\ell \BS (\phi(A)/r)) / o_3$,
so the suffix must contain at least one of them.
Moreover, the chosen suffix cannot start with $\ell$,
as the count of $\ell$ in this prefix is zero.
So the only possibility for the chosen suffix
is to start with $(\ell \BS (\phi(A)/r)) / o_3$. 
To show that such a string is not reducible to
$\ell\BS(\phi(A)/r)$, we need to apply
Proposition~\ref{main_technical_proposition} for 
$A/C = o_3 / (\phi(B)/r) $, $\beta =z'_{A,B}$ and 
$C\BS D = (\phi(A) /r) \BS o_4$.
So one of the following strings has to be reducible to $\ell\BS(\phi(A)/r)$.
\begin{multline*} 
	(\ell \BS (\phi(A)/r)) / o_1;
	o_1 / (\phi(B)/r); \phi(B)/r ; \\ 
	(\phi(A)/r) \BS o_2 ; o_2 \BS (\phi(B)/r);
	(\ell \BS (\phi(A)/r)) \BS (\ell \BS (\phi(B)/r))
\end{multline*}
\begin{multline*}
	(\ell \BS (\phi(A)/r)) / o_1;
	o_1 / (\phi(B)/r); \phi(A)/r ; \\ 
	(\phi(A)/r) \BS o_2 ; o_2 \BS (\phi(B)/r);
	(\ell \BS (\phi(A)/r)) \BS (\ell \BS (\phi(B)/r))
\end{multline*}
The former string is only reducible to
\begin{equation*}
	(\ell \BS (\phi(A)/r)) ;
	(\phi(A)/r) \BS o_2 ; o_2 \BS (\phi(B)/r);
	(\ell \BS (\phi(A)/r)) \BS (\ell \BS (\phi(B)/r))
\end{equation*}
The latter string is reducible to $\ell \BS (\phi(B)/r)$.
\end{proof}

The next step is to generalize $u_{A,B}$
to support more than two categories,
as well as to wrap it with a final layer of divisions
that finally eliminates cross-contamination
of multiple neighbouring encodings.

For any categories $A_1, \ldots, A_n$,
the goal is to construct a string $w_{A_1, \ldots, A_n}$
that would reduce to $\psi(A_1)$, \ldots, $\psi(A_n)$. 
Let $e_1, \ldots e_{2n-2}$ be new primitive categories,
which do not occur in any $u_{A_i,A_{i+1}}$.
Then the desired string is
\begin{align*}
{}& w_{A_1, \ldots, A_n} = x_1; u_{A_1,A_2}; x_2; u_{A_2,A_3}; x_3; \ldots; x_{n-1}; u_{A_{n-1},A_n}; x_n,
\quad \text{where} \\ 
{}& x_1 = \psi(A_1)/e_1; e_1 / \psi(A_2), \\
{}& x_i = \psi(A_{i-1}) \BS e_{2i-2}; (e_{2i-2} \BS \psi(A_i) )/ e_{2i-1}; e_{2i-1} / \psi(A_{i+1}), \\
{}&  \hspace{8cm} \text{for all }  i \in \{2, \ldots, n-1\}, \\
{}& x_n = \psi(A_{n-1}) \BS e_{2(n-1)}; e_{2(n-1)} \BS \psi(A_n).
\end{align*}
If there is only one category on the list,
then $w_{A_1}$ is defined as the string $\psi(A_1)$.

\begin{lemma}\label{w_A1_An_lemma}
Let $A_1, \ldots, A_n$, with $n \geqslant 1$, be any pairwise distinct categories
of the form $p$, $p / q$ or $(p / q) / s$.
Then the string of categories $w_{A_1, \ldots, A_n}$
reduces to the categories $\psi(A_1), \ldots, \psi(A_n)$.
\end{lemma}
\begin{proof}
The case of $n=1$ is clear; for $n \geqslant 2$,
the proof is by induction on $n$.

\textbf{Base case: $n=2$}. Then the string $w_{A_1,A_2}$
is reducible both to $\psi(A_1)$ and $\psi(A_2)$ by the following reductions.
\begin{equation*}
	\infer{\psi(A_1)}{
		\psi(A_1)/e_1
		&
		\infer{e_1}{
			e_1 / \psi(A_2)
			&
			\infer{\psi(A_2)}{
				\infer{e_2}{
					\infer{\psi(A_1)}{u_{A_1, A_2}}
					&
					\psi(A_1) \BS e_2
				}
				&
				e_2 \BS \psi(A_2)
			}
		}
	}
\end{equation*}
\begin{equation*}
	\infer{\psi(A_2)}{
		\infer{e_2}{
			\infer{\psi(A_1)}{
				\psi(A_1)/e_1
				&
				\infer{e_1}{
					e_1 / \psi(A_2)
					&
					\infer{\psi(A_2)}{u_{A_1,A_2}}
				}
			}
			&
			\psi(A_1) \BS e_2
		}
		&
		e_2 \BS \psi(A_2)
	}
\end{equation*}

\textbf{Induction step:}
The following reductions show that the suffix $u_{A_{n},A_{n+1}};x_{n+1}$
of a string $w_{A_1, \ldots, A_{n+1}}$
is reducible to $\psi(A_{n+1})$,
and $x_{n};\psi(A_{n+1})$ is reducible to
$\psi(A_{n-1}) \BS e_{2n-2};e_{2n-2} \BS \psi(A_{n})$:
\begin{equation*}
	\infer{ \psi(A_{n+1}) }{
		\infer{ e_{2n}}{
			\infer{\psi(A_{n})}{u_{A_{n},A_{n+1}}}
			&
			\psi(A_{n}) \BS e_{2n}
		}
		&
		e_{2n} \BS \psi(A_{n+1})
	}
\end{equation*}

\begin{equation*}
	\psi(A_{n-1}) \BS e_{2n-2} \qquad 
	\infer{ e_{2n-2} \BS      \psi(A_{n})  }{
		(e_{2n-2} \BS \psi(A_{n}) )/ e_{2n-1}
		&
		\infer{e_{2n-1}}{
			e_{2n-1} / \psi(A_{n+1})
			&
			\psi(A_{n+1})
		}
	}
\end{equation*}

So the string $w_{A_1, \ldots, A_{n+1}}$,
which is of the form $\ldots;x_n;u_{A_n,A_{n+1}}; x_{n+1}$,
is reducible to the string $\ldots;u_{A_{n-1},A_{n}};x_{n};\psi(A_{n+1})$
and further to $\ldots; u_{A_n,A_{n+1}}; \psi(A_{n-1}) \BS e_{2n-2};e_{2n-2} \BS \psi(A_{n-1})$,
which is exactly the string $w_{A_1, \ldots, A_n}$.
Then, by the induction assumption, the string $w_{A_1, \ldots, A_{n+1}}$
is reducible to $\psi(A_i)$ for $i\in \{1,\ldots,n\}$.

At the same time,
the string $w_{A_1, \ldots, A_{n+1}} = \ldots ;x_n; u_{A_n,A_{n+1}}; x_{n+1}$
is reducible to the string $\ldots ;x_n; \psi(A_{n+1}); x_{n+1}$
(by reducing $u_{A_n,A_{n+1}}$ to $\psi(A_{n+1})$),
and again such a string is reducible to $w_{A_1, \ldots, A_{n}};x_{n+1}$.
Then, by the induction assumption,
the whole string $w_{A_1, \ldots, A_n}$
is reducible to $\psi(A_n);x_{n+1}$.
And the latter string is exactly the string
which was reduced to $\psi(A_{n+1})$ in the first reduction above.
Therefore, the string $w_{A_1, \ldots, A_{n+1}}$ is reducible to $\psi(A_{n+1})$.
\end{proof}

\section{Encoding grammars in grammars with unique category assignment}\label{section_encoding_grammars}

Let $G=(\Sigma, N, R, S)$ be the original context-free grammar
in the Greibach normal form,
and denote the language defined by each nonterminal symbol $X \in N$ by $L_G(X)$.
For every symbol $a \in \Sigma$,
consider all rules with $a$ on the right-hand side,
which are of the form
$X \to a$,
$X \to a Y$ and
$X \to a Z Y$.
Then, the homomorphic image of $a$ is defined as:
\begin{align*}
h(a) &= \ell;w_{A_1, \ldots, A_m};r,
\\ 
& \text{where }
\{A_1, \ldots, A_m\}=\set{X}{X \to a \in R}
\cup
\set{X/Y}{X \to a Y \in R}
\cup 
\\
& \hspace{64mm} \cup \set{(X/Y)/Z}{X \to a Z Y \in R}
\end{align*}

By Lemma~\ref{w_A1_An_lemma}, the string of categories $h(a)$
is reducible to each of the categories $\phi(A_1)$, \ldots, $\phi(A_m)$,
that is, to the encodings of the above categories,
each representing one of the rules of the grammar involving $a$
(or, in other words, one of the possible syntactic roles of $a$).

Then, for every string $w=a_1 \ldots a_n$,
its image $h(a_1 \ldots a_n)$
reduces to a string of categories $B_1 \ldots B_n$,
where each $B_i$ represents one of the syntactic roles of the corresponding symbol $a_i$.
These categories can, so to say, self-assemble
into a parse tree of $w$,
as long as such a tree exists.
The following lemma states this formally,
as well as asserts the converse implication:
reduction to a single category implies the existence of a parse in the grammar $G$.

\begin{lemma}[Main lemma]\label{main_lemma}
For every string $y \in \Sigma^+$
and for every nonterminal symbol $X \in N$,
the string $y$ is in $L_G(X)$
if and only if{\tiny }
its image $h(y)$ is reducible to the category $\phi(X)$.
\end{lemma}

The proof is carried out separately in the two directions,
with ``only if'' being the easier one.

\begin{proof}[Proof in the easy direction]
Let a string $y \in \Sigma^+$ be in $L_G(X)$, for some $X \in N$.
The goal is to construct a reduction of $h(y)$ to the category $\phi(X)$.
This is done inductively on the length of $y$.

\textbf{Base case:} $y=a \in \Sigma$.
Then, there is a rule $X \to a$ in $G$,
and accordingly the category $X$ is listed in the image of $a$, that is,
$h(a)=\ell; w_{\ldots, X, \ldots}; r$.
By Lemma~\ref{w_A1_An_lemma},
the string $w_{\ldots, X, \ldots}$ is reducible to $\psi(X)=\ell \BS (\phi(X)/r)$.
Hence, $h(a)$ is reducible to $\phi(X)$.

\textbf{Induction step:} $|y| \geqslant 2$.
Then, $y$ is obtained by a rule of the form $X \to a Y$ or $X \to a Y Z$.

In the former case, let $y=ay'$,
where $y'$ is defined by the nonterminal symbol $Y \in N$.
Then, by the induction assumption, $h(y')$ reduces to $\phi(Y)$.
Furthermore, as the rule $X \to a Y$ is in $G$, 
the category $X/Y$ is listed in the image $h(a)$,
that is, $h(a)=\ell; w_{\ldots, X/Y, \ldots}; r$.
Accordingly, by Lemma~\ref{w_A1_An_lemma}, the string $h(a)$ is reducible to
$\ell ;\ell \BS (\phi(X/Y)/r) ; r = \ell ;\ell \BS (\phi(X)/\phi(Y))/r); r$,
and then to $\phi(X) / \phi(Y)$.
Therefore, $y$ is reducible to 
$\phi(X) / \phi(Y); \phi(Y)$, which is in turn reducible to $\phi(X)$. 

In the case of a rule $X \to a Y Z$
the string $y$ splits to $y=ay'z$,
where $y' \in L_G(Y)$ and $z \in L_G(Z)$.
Then, by the induction assumption twice,
$h(y')$ is reducible to $\phi(Y)$
and $h(z)$ reduces to $\phi(Z)$.
Again, as the rule $X \to a Y Z$ is listed in the image $h(a)$,
this image is then reducible to $\phi((X/Z)/Y)$ by Lemma~\ref{w_A1_An_lemma}.
The latter category equals $(\phi(X)/\phi(Z))/\phi(Y)$.
Then the whole string $h(x)$ reduces first to
$(\phi(X)/\phi(Z))/\phi(Y); \phi(Y); \phi(Z)$ and then to $\phi(X)$.
\end{proof}

What remains to be proved is the ``hard direction'' of the main lemma, that is, that
whenever the homomorphic image $h(y)$ of a string $y$
is reducible to any category corresponding to a nonterminal symbol of the original grammar,
this implies that $y$ is defined by that nonterminal symbol.
The first part of this proof is the following lemma,
which states that if $h(y)$ reduces to $X/X$,
then there is a reduction that does so in an orderly manner
by separately reducing the image of each symbol to a category
representing one of the syntactic roles of that symbol.

\begin{lemma}\label{orderly_reduction_lemma}
If $h(y) = h(a_1\ldots a_n)$ is reducible to category $X/X$,
then there exists a string of the form 
$\phi(C_1); \ldots; \phi(C_n)$, which is reducible to $X/X$,
where each $C_i$ is drawn from the following set.
\begin{equation*}
        \set{X}{X \to a_i \in R}
	\cup
	\set{X/Y}{X \to a_i Y \in R}
	\cup 
	\set{(X/Y)/Z}{X \to a_i Z Y \in R}
\end{equation*}
\end{lemma}
\begin{proof}

Let $h(y) = h (a_1 \ldots a_n)$ be a string of categories reducible to $X/X$.
Consider the substring
$h(a_1)= \ell; w_{A_1,\ldots,A_m}; r =\ell; x_1; u_{A_1,A_2}; x_2; \ldots; x_m ; r$ of $h(a_1\ldots a_n)$,
where $A_1, \ldots, A_m$ are all categories transcribing the rules of the grammar involving $a_1$,
as described in the definition of $h$. 
So $h(a_1\ldots a_n)$ is string of the following form:
\begin{multline*}
	\ell; \underbrace{\psi(A_1)/e_1 ;e_1 / \psi(A_2)}_{x_1} ; u_{A_1,A_2}; 
		\underbrace{ \psi(A_1) \BS e_2 ; (e_2\BS \psi(A_2))/e_3 ; 
			e_3 / \psi(A_3)}_{x_2} ; u_{A_2, A_3}; x_3; \\ 
		\ldots ;x_m ; r ; h(a_2 \ldots a_n)    
\end{multline*}
By Lemma~\ref{u_A_B_lemma}, no proper prefix of
$u_{A_1,A_2}$ is reducible to $\psi(A_2)$
and no proper suffix of $u_{A_1,A_2}$ is
reducible to $\psi(A_1)$, also $e_1 / \psi(A_2)$ 
and $\psi(A_1) \BS e_2$ are not in the set of denominators
of the whole string, by the definition of $w_{A_1 \ldots A_m}$. 
Then, Proposition~\ref{main_technical_proposition} is applicable to $h(a_1\ldots a_n)$, with the following partition
into $\alpha; A/C; \beta; D \BS B; \gamma$.
\begin{multline*}
	\underbrace{\ell; \psi(A_1)/e_1 }_{\alpha};\underbrace{e_1 / \psi(A_2)}_{A/C} ; \underbrace{u_{A_1,A_2}}_{\beta} ; \underbrace{\psi(A_1) \BS e_2}_{D \BS B} ;\\ \underbrace{(e_2\BS \psi(A_2))/e_3 ; e_3 / \psi(A_3) ;  u_{A_2,A_3} ;  x_3; \ldots ;x_m ; r ; h(a_2 \ldots a_n) }_{\gamma}  
\end{multline*}
So, by Proposition~\ref{main_technical_proposition} 
for $A/C = e_{1}/\psi(A_2)$, 
$\beta = u_{A_1,A_2}$ and $D\BS B=\psi(A_1)\BS e_{2}$,
one of the strings obtained from $h(y)$
by replacing $u_{A_1,A_2}$ with $\psi(A_1)$ or $\psi(A_2)$---that is, one of the following two strings---is reducible to the category $X/X$:
\begin{align*}
	& \psi(A_1)/e_1 ;e_1 / \psi(A_2) ; \underline{\psi(A_1)} ; \psi(A_1) \BS e_2 ;
	(e_2\BS \psi(A_2))/e_3 ; e_3 / \psi(A_3) ; \ldots ;x_m ; r ; \ldots \\ 
	& \psi(A_1)/e_1 ;e_1 / \psi(A_2) ; \underline{\psi(A_2)} ; \psi(A_1) \BS e_2 ;
	(e_2\BS \psi(A_2))/e_3 ; e_3 / \psi(A_3) ; \ldots ;x_m ; r ; \ldots     
\end{align*}
Let $j_1 \in \{1, 2\}$ be such that the $j_1$-th of these strings
(the one in which $u_{A_1,A_2}$ is replaced with $\psi(A_1)$)
is reducible to $X/X$.
Consider the substring $u_{A_2,A_3}$ of that string.
The goal is to apply Proposition~\ref{main_technical_proposition} to that string, with the following partition.
\begin{multline*}
	\underbrace{\ell; \psi(A_1)/e_1 ;e_1 / \psi(A_2) ; \psi(A_{j_1}) ; \psi(A_{i}) \BS e_2 ; (e_2\BS \psi(A_2))/e_3 }_{\alpha};\\
	\underbrace{e_3 / \psi(A_3)}_{A/C} ; \underbrace{u_{A_2,A_3}}_{\beta} ; \underbrace{\psi(A_2) \BS e_4}_{D\BS B} ; 
	\underbrace{\ldots ;x_m ; r ; h(a_2 \ldots a_n) }_{\gamma}  
\end{multline*}
Then, Proposition~\ref{main_technical_proposition} gives two new strings,
in which $u_{A_2,A_3}$ is replaced with $\psi(A_2)$ and with $\psi(A_3)$, respectively,
and one of these strings is reducible to $X/X$.
Let $j_2 \in \{2, 3\}$ be such that
it is the string with $u_{A_2,A_3}$ replaced with $\psi(A_{j_2})$
that is reducible to $X/X$.

The same argument is repeated for each substring $u_{A_k,A_{k+1}}$ of $h(a_1)$,
in total $m-1$ times, until $h(a_1 \ldots a_n)$ is reduced
to a string of the form $\ell;\alpha_1;r;h(a_2);\ldots;h(a_n)$,
which is also reducible to $X/X$, where $\alpha_1$ is a string 
obtained from $w_{A_1, \ldots A_n}$ by reducing each 
$u_{A_i,A_{i+1}}$ either to category $\psi(A_i)$ or to $\psi(A_{i+1})$. 
	
Now we can consider substring $h(a_2)$ in the string 
$\ell;\alpha_1;r;h(a_2);\ldots;h(a_n)$ and repeat the 
above line of reasoning for it, and so on for each $h(a_i)$.
At last we obtain a string of the form $\ell;\alpha_1;r; \ldots; \ell;\alpha_n;r$,
which is reducible to $X/X$, 
where each $\alpha_i$ is obtained by 
reducing all $u_{C,D}$ in $h(a_i)$, each to a single category.
	
Consider any $\alpha_i$ in the resulting string.
This substring consists of categories of two types.
First, there are categories
obtained by the reduction of some $u_{D,E}$ to $\psi(D)$ or $\psi(E)$,
and since each $u_{D,E}$ is a substring of $h(a_i)$,
by the definition of $h(a_i)$,
these are categories of the form $\psi(C)$,
where $C$ is in the following set.
\begin{equation*}
	\set{X}{X \to a_i \in R}
	\cup
	\set{X/Y}{X \to a_i Y \in R}
	\cup
	\set{(X/Y)/Z}{X \to a_i Z Y \in R}
\end{equation*} 
The other kind of categories in $\alpha_i$
are all the remaining categories from
the original string $h(a_i)=\ell; w_{A_1, \ldots, A_n}; r$,
which are all preserved in $\ell; \alpha_i; r$.
Each of these categories is of the form 
$\psi(C)/e_k$,
$e_k/\psi(C)$,
$\psi(C) \BS e_k$,
$e_k \BS \psi(C)$ or
$(e_{k-1}\BS \psi(C))/e_k$,
where $C$ is from the above set of categories
(that rephrase the rules involving $a_i$),
and $e_{k-1}$ and $e_k$
are primitive categories introduced in the definition of $w_{A_1, \ldots, A_n}$.
All kinds of categories that occur in $\alpha_i$ have the following common properties.
	
\begin{ourclaim}\label{claim_alpha_1}
Each category in $\alpha_i$ has equal count of $\ell$ and of $r$.
The only categories occurring among the numerators of $\alpha_i$
that can be reduced with $\ell$ are the categories $\psi(C)$,
where $C$ is from the set above.
Similarly, the only categories to be reduced with $r$ are $\phi(C)/r$,
for $C$ from the same set.
\end{ourclaim}
The proof is by examining all forms of these categories listed above.
Hence, for any category obtained by a reduction of some categories in $\alpha_i$,
the count of $\ell$ in this category is equal to the count of $r$ in it.

\begin{ourclaim}\label{claim_alpha_4}
In the reduction of the string 
$\ell;\alpha_1;r;\ldots;\ell;\alpha_n;r$ 
to the category $X/X$,
no categories of the following forms 
can be obtained:
$D/\ell$, $r \BS D$,  $D/(\phi(C)/r)$ or 
$(\phi(C)/r) \BS D$,
where $C$ and $D$ are any categories.
\end{ourclaim}

This is so, because all the categories of such forms
are not in the set of numerators of the whole string
$\ell;\alpha_1;r; \ldots; \ell;\alpha_n;r$,
which can be seen by examining all forms of categories 
in the set of numerators of $\alpha_i$.

Now it is claimed that in 
the reduction of $\ell;\alpha_1;r; \ldots; \ell;\alpha_n;r$ to category $X/X$,
every substring $\ell;\alpha_i;r$
has to be first reduced to a string of the form $\phi(C) / r; \widetilde{x};r$.
	
Consider the last step in this reduction
when the categories $\ell$ and $r$ from the substring $\ell;\alpha_i;r$
still remain intact.
Let $\ell;x;r$ be the current substring derived from $\ell;\alpha_i;r$.
Then, at the next step,
one of categories $\ell$ or $r$
is reduced with some category $B$.
Then, first of all, the category $B$
must be inside the substring $x$---
for otherwise, if $B$ is to the left of $\ell$ or to the right of $r$,
it would have to be of the form $D/ \ell$ or $r \BS D$,
but such categories cannot be obtained in the reduction
of the whole string to $X/X$ by Claim~\ref{claim_alpha_4}.
Hence, $B$ is one of the categories in the substring $x$,
and so it must be the result of reduction of some categories in 
$\alpha_i$.
Then, by Claim~\ref{claim_alpha_1},
the count of $\ell$ and $r$ in the category $B$ are equal. 
At the same time, $B$ has to 
be reducible with category $\ell$ or $r$, 
and so, by Claim~\ref{claim_alpha_1},
$B = \ell \BS (\phi(C) / r)$ or $B = \phi(C) / r$ for some $C$;
note that the latter case is not possible,
because the actual count of $r$ in such a category is $-1$
and the count of $\ell$ is $0$. 
Then, in the reduction of the whole string to $X/X$,
the substring $\ell;\alpha_i;r$
must have been reduced to $\ell; \ell \BS (\phi(C) / r); \widetilde{x};r$,
where $x=\ell \BS (\phi(C) / r); \widetilde{x}$,
and then further reduced to $\phi(C) / r; \widetilde{x};r$.
	
\begin{ourclaim}
The string $\widetilde{x}$ must be empty,
and hence the substring $\ell; \alpha_i; r$
must have been reduced to $\phi(C) / r; r$.
\end{ourclaim}
	
Suppose the contrary, that is, $\widetilde{x}$ is not empty.
Consider the first reduction, in which one of the categories $r$ or $\phi(C)/r$
from the resulting substring $\phi(C) / r; \widetilde{x};r$ is involved.
These categories cannot be reduced with any categories not from the substring $\widetilde{x}$,
because for that the other category would have to be 
of the form $r \BS D$ or $D / (\phi(C) / r)$;
however, by Claim~\ref{claim_alpha_4},
such categories cannot occur in any reduction of the whole string. 
	
Therefore, $r$ or $\phi(C)/r$ was reduced by 
some category obtained
by a reduction of some substring of $\widetilde{x}$. 
Let it be a category $E$ which was reduced with $r$ or $\phi(C)/r$.
By Claim~\ref{claim_alpha_1},
the count of $\ell$ in this category $E$ is equal to the count of $r$ in this category.
Moreover, by Claim~\ref{claim_alpha_1}, the category $E$
could only be reduced with $r$ if $E=\phi(C)/r$;
but this cannot be the case,
since $\#_{\ell}(\phi(C)/r) = 0$ and $\#_{r}(\phi(C)/r) = -1$.
On the other hand, 
to reduce $E$ with $\phi(C)/r$,
it has to be of the form $(\phi(C)/r)\BS D$ or $r$.
By Claim~\ref{claim_alpha_4}, the category $E$
cannot be of the form $(\phi(C)/r)\BS D$, 
and $E$ cannot be $r$, because $\#_{\ell}(E) = \#_{r}(E)$,
but $\#_{\ell}(r) = 0$ and $\#_{r}(r) = 1$.
	
Therefore, in the reduction of $\ell;\alpha_1;r;\ldots;\ell;\alpha_n;r$ 
to a category $X/X$,
each block $\ell;\alpha_i;r$
was first reduced to a string of the form $\phi(C_i) / r; r$,
where every $C_i$ is in the following set.
\begin{equation*}
	\set{X}{X \to a_i \in R}
	\cup
	\set{X/Y}{X \to a_i Y \in R}
	\cup
	\set{(X/Y)/Z}{X \to a_i Z Y \in R}
\end{equation*}
Similarly to the first part of this proof,
one can infer from this that there is a string of categories of the form
$\phi(C_1) / r; r; \ldots; \phi(C_n) / r; r$,
where each $C_i$ is from the corresponding set defined above,
which is reducible to $X/X$ as well.
	
Finally, the reduction of the latter string $\phi(C_1) / r; r; \ldots; \phi(C_n) / r; r$ to $X/X$
has to reduce each block $\phi(C_i) / r ; r$ to $\phi(C_i)$,
because the categories $r$ cannot be eliminated in any other way.
Therefore, $\phi(C_1); \ldots; \phi(C_n)$ is reducible to $X/X$, as claimed.
\end{proof}

This confirms that the reduction of $h(y)$ to $X/X$
must proceed in the orderly manner,
which is used to complete the proof of Lemma~\ref{main_lemma}.

\begin{proof}[Proof of Lemma~\ref{main_lemma} in the hard direction]
	Let $y=a_1 \ldots a_n \in \Sigma^+$,
	and assume that $h(y)$ is reducible to a category $\phi(X) = X/X$ for some nonterminal $X$.
	It is claimed that the string $y$ must be in $L_G(X)$.
	The proof is by induction on the length of $y$.
	
	\textbf{Base case:} $y = a \in \Sigma$.
	Since $h(a)$ is reducible to a single category $X/X$ by the assumption,
	by Lemma~\ref{orderly_reduction_lemma},
	there exists a one-symbol string of categories $\phi(C)$,
	where 
	$C \in \set{X}{X \to a \in R}
	\cup
	\set{X/Y}{X \to a Y \in R}
	\cup
	\set{(X/Y)/Z}{X \to a Z Y \in R}$,
	which is reducible to $X/X$.
	But since $X/X$ is already a single category,
	this reduction is only possible if $\phi(C)=X/X$, and hence $X=C$.
	Then there is a rule $X \to a \in R$,
	and hence $a$ is in $L_G(X)$, as claimed.
	
\textbf{Induction step:} 
$y=a_1 \ldots a_m$, with $m \geqslant 2$.
As in the base case, since $h(y)$ 
reduces to a single category $X/X$,
by Lemma~\ref{orderly_reduction_lemma},
there exists a string of 
categories $\phi(C_1) \ldots \phi(C_n)$,
which is reducible to $X/X$,
and every $C_i$ belongs to the corresponding set
\begin{equation*}
\set{X}{X \to a_i \in R}
\cup
\set{X/Y}{X \to a_i Y \in R}
\cup \\
\set{(X/Y)/Z}{X \to a_i Z Y \in R}
\end{equation*}

Due to this form of each $C_i$,
with all divisions to the right,
any reductions applied to the 
string $\phi(C_1) \ldots \phi(C_n)$
are of the form $E/F;F \to E$.
This, in particular, means that the resulting category $X/X$
is in the numerator of $\phi(C_1)$,
and there are two possibilities:
either $\phi(C_1) = \phi(X) / \phi(Y)$,
or $\phi(C_1) = (\phi(X) / \phi(Y))/ \phi(Z)$.
	
In the first case, the whole string 
$\phi(C_2) \ldots \phi(C_n)$ is reduced to $\phi(Y)$,
and, by the induction assumption,
the string $a_2 \ldots a_m$ is in $L_{G}(Y)$. 
Moreover, as $\phi(C_1) = \phi(X) / \phi(Y) = \phi(X/Y)$,
the category $C_1$ is $X/Y$.
Hence, there is a rule $X \rightarrow a_1 Y$ 
in the original grammar,
and then $a_1 \ldots a_n \in L_{G}(X)$ by this rule.
	
For the second case, 
if $\phi(C_1) = (\phi(X) / \phi(Y))/ \phi(Z)$,
then the reduction of 
$\phi(C_1);\ldots;\phi(C_m)$ to $\phi(X)$ 
is concluded with the following two steps.
\begin{equation*}
\infer{\phi(X)}{
	\infer{\phi(X) / \phi(Y)}{
		(\phi(X) / \phi(Y))/\phi(Z)
		&
		\phi(Z)
	}
	&
	\phi(Y)
}
\end{equation*}
Accordingly, the string $\phi(C_2) \ldots \phi(C_n)$ 
is split into two substrings,
where the first part $\phi(C_2) \ldots \phi(C_i)$
is reduced to $\phi(Z)$,
and the second part $\phi(C_{i+1}) \ldots \phi(C_n)$
is reduced to $\phi(Y)$.
Then, by the induction assumption,
$a_2 \ldots a_i \in L_G(Z)$ and
$a_{i+1} \ldots a_m \in L_G(Y)$.
Since 
$\phi(C_1)=(\phi(X) / \phi(Y))/ \phi(Z) = \phi((X/Y)/Z)$
then the category $C_1$ is $(X/Y)/Z$.
Hence, there is a rule $X \rightarrow a_1 Z Y$ 
in the original grammar,
and then $a_1 \ldots a_n \in L_{G}(X)$ by this rule.
\end{proof}

Altogether, the following theorem has been proved.

\begin{theorem}
For every context-free language $L \subseteq \Sigma^+$
there exists an alphabet $\Omega$,
a homomorphism $h \colon \Sigma \to \Omega^+$,
and a categorial grammar $G=(\Omega, \PC, \triangleright, S)$
with a unique category assignment,
such that a string $w$ is in $L$
if and only if $h(w)$ is in $L(G)$.
\end{theorem}

Among its immediate implications,
there is the following strengthened statement
of Greibach's~\cite{Greibach_hardest} hardest language theorem.

\begin{corollary}
There exists an alphabet $\Omega$
and a language $L_0 \subseteq \Omega^*$
defined by a categorial grammar with unique category assignment,
such that for every context-free language $L \subseteq \Sigma^+$
there exists a homomorphism $h \colon \Sigma \to \Omega^+$,
such that $w$ is in $L$
if and only if $h(w)$ is in $L_0$.
\end{corollary}

Unique category assignment intuitively means that every symbol is unambiguously interpreted, but does not guarantee that the entire parse is unambiguous.
Can these grammars define any inherently ambiguous languages? They can.
\begin{corollary}
There exists an inherently ambiguous language
defined by a categorial grammar with unique category assignment.
\end{corollary}
\begin{proof}
Let $L$ be any inherently ambiguous language,
such as $L=\set{a^i b^n c^n}{i,n \geqslant 1} \cup \set{a^m b^m c^j}{m,j \geqslant 1}$,
and let $h$ and $G$ be given for $L$ by the theorem.
Then $L(G)$ is inherently ambiguous
~\cite[Thm.~7.4.2, Cor.~1]{ItFLT_Harrison},
for if it were unambiguous,
then so would be its inverse homomorphic image $h^{-1}(L(G))=L$,
which is not the case. 
\end{proof}

\section{Conclusion}

A suggested topic for future research
is to consider the restriction of unique category assignment
in other variants of categorial grammars.
Some of these variants have already been considered:
Safiulin~\cite{Safiulin} showed that
the Lambek calculus with unique category assignment
is as powerful as the Lambek calculus of the general form
(that is, they can define every context-free language).
Recently, Kuznetsov~\cite{Kuznetsov_prodfree} established an analogous result
for product-free Lambek calculus.
But for some natural categorial formalisms,
the same question is yet to be investigated:
for instance, nothing is known
about the expressive power of combinatory categorial grammars
with unique category assignment,
and conjunctive categorial grammars~\cite{Kuznetsov,KuznetsovOkhotin}
(that is, categorial grammars extended with conjunction)
are in the same situation.

\end{document}